\documentclass[11pt,reqno]{amsart}
\usepackage[T1]{fontenc}
\usepackage{lmodern}
\usepackage{amsmath,amssymb,mathtools}
\usepackage{enumitem}
\usepackage{microtype}
\usepackage[a4paper,margin=22mm]{geometry}
\usepackage[draft]{hyperref} 
\usepackage{tikz-cd}

\newtheorem{theorem}{Theorem}[section]
\newtheorem{lemma}[theorem]{Lemma}
\newtheorem{proposition}[theorem]{Proposition}
\newtheorem{corollary}[theorem]{Corollary}
\newtheorem{conjecture}[theorem]{Conjecture}
\theoremstyle{definition}
\newtheorem{definition}[theorem]{Definition}
\theoremstyle{remark}

\theoremstyle{plain}

\DeclareMathOperator{\Hom}{Hom}
\DeclareMathOperator{\End}{End}
\DeclareMathOperator{\Ext}{Ext}
\DeclareMathOperator{\Tor}{Tor}
\DeclareMathOperator{\add}{add}
\DeclareMathOperator{\domdim}{domdim}
\DeclareMathOperator{\pdim}{pdim}

\DeclareMathOperator{\rad}{rad}

\newcommand{\op}{\mathrm{op}}
\newcommand{\modu}{\operatorname{mod}}
\newcommand{\Db}{\mathrm{D}^{\mathrm b}}
\newcommand{\NC}{\mathrm{NC}}
\newcommand{\TCone}{\mathrm{TC1}}
\newcommand{\TCtwo}{\mathrm{TC2}}
\newcommand{\ARC}{\mathrm{ARC}}

\numberwithin{equation}{section}
\setlist[enumerate]{label=\textup{(\arabic*)},leftmargin=*,itemsep=2pt}
\allowdisplaybreaks[1]
\title[On the homological conjectures for Artin algebras]{On the homological conjectures for Artin algebras}
\author[H. Enomoto]{Haruhisa Enomoto}
\address{Haruhisa Enomoto: Parakeet Inc., Japan}
\email{haruhisa.enomoto.math@gmail.com}
\author[R. Marczinzik]{Rene Marczinzik}
\address{Rene Marczinzik: Mathematical Institute of the University of Bonn,
Endenicher Allee 60, 53115 Bonn, Germany}
\email{marczire@math.uni-bonn.de}
\date{\today}

\begin{document}
\begin{abstract}
We prove that the Auslander--Reiten conjecture, the generalised Nakayama
conjecture, the Nakayama conjecture, and each of the two Tachikawa
conjectures are globally equivalent for Artin algebras over any fixed
commutative artinian ring. For finite dimensional algebras over a
fixed field, we also prove that the derived invariance of the Nakayama
conjecture is equivalent to the truth of the Nakayama conjecture for all such algebras, answering a question raised by Chen and Xi.
\end{abstract}
\maketitle

\section{Introduction}
In this article we give a new viewpoint on the longstanding homological conjectures for Artin algebras.
For background and historical accounts, see the surveys of
Happel~\cite{Hap91}, Yamagata~\cite{Yam96}, and
Huisgen-Zimmermann~\cite{HZ95}.

Throughout the first three sections, $R$ is a fixed commutative artinian
ring, algebras are Artin $R$-algebras, and modules are finitely generated
right modules. 
The \emph{finitistic dimension} of $A$ is
\[
\operatorname{findim}A
=\sup\{\pdim_A M\mid M\in\modu A,\ \pdim_A M<\infty\}.
\]
For an algebra $A$ with minimal injective coresolution
\[
0\longrightarrow A_A\longrightarrow I^0\longrightarrow I^1
\longrightarrow\cdots,
\]
the \emph{dominant dimension} $\domdim A$ is the least $n\geq0$ such
that $I^n$ is not projective, and is infinite if all the terms are
projective. In particular, a selfinjective algebra has infinite dominant
dimension. The \emph{grade} of a module $M$ is
\[
\operatorname{grade}_A M
=\inf\{i\geq0\mid\Ext_A^i(M,A)\neq0\},
\qquad\inf\varnothing=\infty.
\]
Following G\'elinas~\cite[Definition~1.1]{Gel22}, the \emph{depth} of
$A$ is the supremum of the grades of its simple right modules.
Since there are only finitely many isomorphism classes of simple modules,
this depth is finite exactly when each simple module has finite grade.

We next recall the four most important homological conjectures related to finiteness of homological dimensions for Artin algebras. The numbers in parentheses refer
to the list of homological conjectures in \cite[p.~410]{ARS95}.
\begin{enumerate}
\item The \emph{finitistic dimension conjecture} (FDC): every Artin
algebra has finite finitistic dimension (Conjecture~(11)).
\item The \emph{Gorenstein symmetry conjecture} (GSC): the left regular
module of an Artin algebra has finite injective dimension if and only
if the right regular module does (Conjecture~(13)).
\item The \emph{generalised Nakayama conjecture} (GNC): every Artin
algebra has finite depth (Conjecture~(9)). 
\item The \emph{Nakayama conjecture} (NC): every non-selfinjective Artin
algebra has finite dominant dimension (Conjecture~(8)).
\end{enumerate}
Many of those conjectures are quite old. For example the Nakayama conjecture dates back to the 1950's, see \cite{Nak58}.
The following four further prominent conjectures will be relevant:
\begin{enumerate}
\item The \emph{strong Nakayama conjecture} (SNC): every nonzero module
has finite grade (Conjecture~(12) in \cite{ARS95}).
\item The \emph{Auslander--Reiten conjecture} (ARC): if
$\Ext_A^i(M,M\oplus A)=0$ for every $i>0$, then $M$ is projective
(Conjecture~(10) in \cite{ARS95}).
\item The \emph{first Tachikawa conjecture} (TC1): if
$\Ext_A^i(DA,A)=0$ for every $i>0$, then $A$ is selfinjective
\cite[\S8]{Tac73}.
\item The \emph{second Tachikawa conjecture} (TC2): if $A$ is
selfinjective and $\Ext_A^i(M,M)=0$ for every $i>0$, then $M$ is
projective \cite[\S8]{Tac73}.
\end{enumerate}

An implication or equivalence between conjectures is called
\emph{global} if it concerns their validity for all Artin $R$-algebras
over the fixed ring $R$; TC2 is quantified over the selfinjective ones.
We remark, that global equivalence need not be an equivalence for each individual algebra.
For example, NC holds for every local Artin algebra, but the corresponding
Tachikawa questions remain open in general even for commutative
artinian local rings. For recent examples concerning low-degree Ext
vanishing in that setting, see \cite{BM26}.
For related open conjectures over commutative local rings inspired by the conjectures in the Artin algebra case, see for example
\cite{ABS05} and \cite{HL04}.

Auslander and Reiten proved the global equivalence of ARC and
GNC~\cite{AR75}. Classical results also give the global equivalence
of NC with the conjunction of TC1 and TC2; see
\cite{Mue68,Tac73}. The following familiar implications provide the
context for our results; see \cite[p.~410]{ARS95} and
\cite[\S3.4]{Yam96}:
\[
\begin{tikzcd}[row sep=small]
 & \mathrm{FDC} & \\
 & \mathrm{SNC} & \mathrm{GSC} \\
 & \mathrm{GNC}\ \Longleftrightarrow\ \mathrm{ARC} & \\
 & \mathrm{NC} & \\
 \mathrm{TC1} && \mathrm{TC2}
 \arrow[Rightarrow,from=1-2,to=2-2]
 \arrow[Rightarrow,from=1-2,to=2-3]
 \arrow[Rightarrow,from=2-2,to=3-2]
 \arrow[Rightarrow,from=3-2,to=4-2]
 \arrow[Rightarrow,from=4-2,to=5-1]
 \arrow[Rightarrow,from=4-2,to=5-3]
\end{tikzcd}
\]
All arrows in this diagram are global implications.

Our first main result shows that each Tachikawa conjecture separately
implies ARC.
\begin{theorem}\label{mainresult1}
Fix a commutative artinian ring $R$. The following conjectures are
globally equivalent for Artin $R$-algebras:
\begin{enumerate}
\item the Auslander--Reiten conjecture;
\item the generalised Nakayama conjecture;
\item the Nakayama conjecture;
\item the first Tachikawa conjecture;
\item the second Tachikawa conjecture.
\end{enumerate}
\end{theorem}
We will prove this theorem using constructions related to $r$-fold trivial extension algebras to show that the first Tachikawa conjecture implies the Auslander-Reiten conjecture globally and then also to show that the second Tachikawa conjecture implies the Auslander-Reiten conjecture globally.
Having proved this, the above diagram looks as follows and proves theorem \ref{mainresult1}:
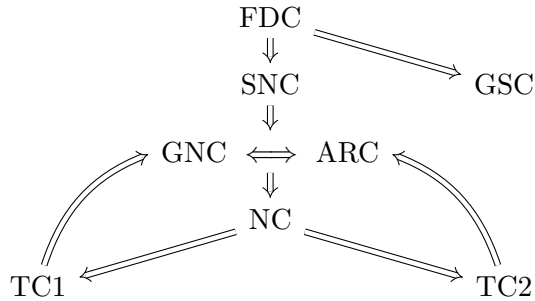
\[
\begin{tikzcd}[row sep=small]
 & \mathrm{FDC} & \\
 & \mathrm{SNC} & \mathrm{GSC} \\
 & \mathrm{GNC}\ \Longleftrightarrow\ \mathrm{ARC} & \\
 & \mathrm{NC} & \\
 \mathrm{TC1} && \mathrm{TC2}
 \arrow[Rightarrow,from=1-2,to=2-2]
 \arrow[Rightarrow,from=1-2,to=2-3]
 \arrow[Rightarrow,from=2-2,to=3-2]
 \arrow[Rightarrow,from=3-2,to=4-2]
 \arrow[Rightarrow,from=4-2,to=5-1]
 \arrow[Rightarrow,from=4-2,to=5-3]
 \arrow[Rightarrow,from=5-1,to=3-2,
        bend left=25,end anchor=west]
 \arrow[Rightarrow,from=5-3,to=3-2,
        bend right=25,end anchor=east]
\end{tikzcd}
\]

Thus we can simplify it to the following diagram, where up to global equivalence only 4 homological conjectures remain:
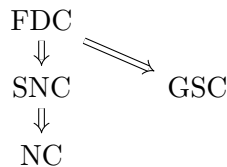
\[
\begin{tikzcd}[row sep=small]
\mathrm{FDC} & \\
\mathrm{SNC} & \mathrm{GSC} \\
\mathrm{NC}
\arrow[Rightarrow,from=1-1,to=2-1]
\arrow[Rightarrow,from=1-1,to=2-2]
\arrow[Rightarrow,from=2-1,to=3-1]
\end{tikzcd}
\]
Here NC represents any of the five globally equivalent conjectures
in Theorem~\ref{mainresult1}, we picked the Nakayama conjecture as a representative as it might be the historically most important one.

Our second application concerns finite dimensional algebras over a
fixed field $K$. Chen and Xi formulated the following conjecture
\cite[\S5.2, conjecture~(1)]{CX16}; see also
\cite[Conjecture~6.5]{Xi18}.
\begin{conjecture}[Chen--Xi]\label{conj:CX}
If $A$ and $B$ are derived equivalent finite dimensional $K$-algebras,
then $\domdim A<\infty$ if and only if $\domdim B<\infty$.
\end{conjecture}
Write $\NC(A)$ for the implication that $\domdim A=\infty$ forces
$A$ to be selfinjective.
\begin{theorem}\label{mainresult2}
For finite dimensional algebras over a fixed field $K$, the following
are equivalent:
\begin{enumerate}
\item NC holds for every such algebra;
\item $\NC(A)$ and $\NC(B)$ are equivalent whenever $A$ and $B$ are
derived equivalent;
\item the Chen--Xi conjecture holds.
\end{enumerate}
\end{theorem}

\section{Preliminaries}\label{sec:prelim}
$R$ is a fixed
commutative artinian ring, $E$ is an injective envelope of $R/\rad R$, and
$D=\Hom_R(-,E)$. An algebra means an associative unital Artin $R$-algebra, finite as an
$R$-module. 
All modules are finitely generated right modules unless otherwise stated.
We write $\modu A$ for the module category of finitely generated $A$-modules and $\add M$ for the full subcategory
of direct summands of finite direct sums of copies of $M$ for an $A$-module $M$.
We refer to \cite{ASS06,SY11,ARS95} for standard textbooks on representation theory of Artin algebras.
The Nakayama functor is given by
\[
\nu_A=D\Hom_A(-,A)\cong-\otimes_A DA.
\]
A module $M$ is \emph{selforthogonal} if $\Ext_A^i(M,M)=0$ for every $i>0$ and an algebra is \emph{selfinjective} if its right regular module is injective.
A module $V$ is a \emph{generator} if $A\in\add V$, and a \emph{cogenerator} if
$DA\in\add V$. Thus a \emph{generator-cogenerator} contains every projective
and every injective module in its additive closure.
An \emph{additive generator} of a subcategory $\mathcal C$ is a module
$Q$ with $\mathcal C=\add Q$. 
Endomorphism algebras are multiplied by composition: $ab=a\circ b$. Thus $\Hom_A(V,N)$ is a right $\End_A(V)$-module by precomposition, while $V$ is a left $\End_A(V)$-module by evaluation.

\begin{lemma}\label{lem:induction}
Let $\Gamma$ be an algebra, let $e\in\Gamma$ be an idempotent and put $A=e\Gamma e$. Consider
\[
F=-\otimes_A e\Gamma:\modu A\longrightarrow\modu\Gamma,
\qquad H=(-)e:\modu\Gamma\longrightarrow\modu A.
\]
Then the following hold.
\begin{enumerate}
\item $F$ is fully faithful and preserves and reflects projective modules.
\item If $\Tor_i^A(M,e\Gamma)=0$ for every $i>0$, then, for every $Z\in\modu\Gamma$ and $n\geq0$, there is a natural isomorphism
\begin{equation}\label{eq:indExt}
\Ext_\Gamma^n(FM,Z)\cong\Ext_A^n(M,Ze).
\end{equation}
\end{enumerate}
\end{lemma}
\begin{proof}
See \cite[Lemma~3.1(1),(3) and pp.~678--679]{APT92} for (1), and
\cite[Theorem~3.2(c) and Proposition~3.7(b)]{APT92} for (2).
\end{proof}

\begin{lemma}\label{lem:yoneda}
Let $V$ be an $S$-module and $B=\End_S(V)$. Then
\[
\Hom_S(V,-):\add V\longrightarrow\add B
\]
is an equivalence. 
\end{lemma}
\begin{proof}
See for example \cite[Proposition~II.2.1(c), p.~33]{ARS95}.
\end{proof}
For the next Proposition, we recall
 M\"uller's formula \cite[Lemma~3, pp.~401--402]{Mue68} that for a generator-cogenerator $V$ we have:
\begin{equation}\label{eq:mueller-form}
\domdim\End_S(V)
 =1+\inf\{i\geq1:\Ext_S^i(V,V)\ne0\},
\qquad\inf\varnothing=\infty.
\end{equation}
\begin{proposition}\label{prop:endExt}
Let $S$ be an algebra, let $V=S\oplus W$ be an $S$-module and put $B=\End_S(V)$. If $V$ is selforthogonal, then there are isomorphisms
\begin{equation}\label{eq:endExt}
\Ext_B^n(DB,B)\cong\Ext_S^n(\nu_S V,V)
\qquad(n\geq0).
\end{equation}
If $V$ is also a cogenerator, then $\domdim B=\infty$.
\end{proposition}

\begin{proof}
Write $H=\Hom_S(V,-)$, and let $e\in B$ be the idempotent corresponding to the summand $S$. There are bimodule isomorphisms
\begin{equation}\label{eq:cornerV}
{}_B V_S\cong {}_B Be_S,
\qquad {}_S eB_B\cong {}_S\Hom_S(V,S)_B.
\end{equation}
Here $eBe=\End_S(S)$ is identified with $S$ by left multiplication. More explicitly, if $\iota:S\to V$ and $\pi:V\to S$ are the inclusion and projection, then $e=\iota\pi$. The first map in \eqref{eq:cornerV} sends $be$ to $b\iota(1)$, and the second sends $eb$ to $\pi b$.

Choose an injective resolution
\[
0\longrightarrow V\longrightarrow I^0\longrightarrow I^1\longrightarrow\cdots.
\]
Since $\Ext_S^i(V,V)=0$ for $i>0$, applying $H$ gives an exact sequence
\begin{equation}\label{eq:Hresolution}
0\longrightarrow B\longrightarrow H(I^0)\longrightarrow H(I^1)\longrightarrow\cdots.
\end{equation}
There is a natural isomorphism of right $B$-modules
\[
H(DS)=\Hom_S(V,DS)\cong DV\cong D(Be).
\]
The first isomorphism sends $f$ to the functional $v\mapsto f(v)(1)$; its inverse sends $\lambda\in DV$ to the map $v\mapsto(s\mapsto\lambda(vs))$. Since $Be$ is projective as a left $B$-module, $D(Be)$ is injective as a right $B$-module. Thus all terms $H(I^j)$ in \eqref{eq:Hresolution} are injective.

For any $S$-module $I$, tensor--Hom adjunction gives
\begin{align*}
\Hom_B(DB,H(I))
&=\Hom_B\bigl(DB,\Hom_S(V,I)\bigr)\\
&\cong\Hom_S(DB\otimes_B V,I)\\
&\cong\Hom_S(DB\otimes_B Be,I)\\
&\cong\Hom_S((DB)e,I).
\end{align*}
Since $(DB)e \cong D(eB)$ and
\[
D(eB)
\cong D\Hom_S(V,S)=\nu_S V
\]
we can identify the entire cochain complexes
\[
\Hom_B(DB,H(I^\bullet))
\cong\Hom_S(\nu_S V,I^\bullet).
\]
The first complex computes $\Ext_B^n(DB,B)$ by \eqref{eq:Hresolution}; the second computes $\Ext_S^n(\nu_S V,V)$. This proves \eqref{eq:endExt}.

Finally, if $V$ is a cogenerator, then $\domdim B= \infty$ by \eqref{eq:mueller-form}, since $V$ is selforthogonal.
\end{proof}

\begin{lemma}\label{lem:splitgenerator}
Let $V=S\oplus W$ be a generator-cogenerator over an algebra $S$, and put $B=\End_S(V)$. If $B$ is selfinjective, then $V$ is injective. If $S$ and $B$ are selfinjective, then $V$ is projective-injective.
\end{lemma}

\begin{proof}
Let $j:V\hookrightarrow I$ be an injective envelope. Since $V$ is a cogenerator, $I\in\add DS\subseteq\add V$. Applying $H=\Hom_S(V,-)$ gives a monomorphism $H(j):B\hookrightarrow H(I)$. As $B_B$ is injective, this has a retraction $r:H(I)\to B$. Lemma~\ref{lem:yoneda} gives $s:I\to V$ with $H(s)=r$. It follows that $H(sj)=1_{H(V)}$, and faithfulness gives $sj=1_V$. Thus $V$ is a direct summand of $I$ and is injective. Over a selfinjective algebra, finitely generated injective and projective modules coincide.
\end{proof}

\begin{corollary}\label{cor:criterion}
Let $S$ be selfinjective, let $V=S\oplus W$, and suppose that
\[
\Ext_S^i(V,V)=\Ext_S^i(\nu_S V,V)=0\qquad(i>0).
\]
If $\End_S(V)$ satisfies $(\TCone)$, then $V$ is projective.
\end{corollary}

\begin{proof}
Proposition~\ref{prop:endExt} and $(\TCone)$ make $\End_S(V)$ selfinjective. Since $S$ is selfinjective, $V$ is a generator-cogenerator. Now apply Lemma~\ref{lem:splitgenerator}.
\end{proof}

\section{Proof of the main result}\label{sec:extensions}

We first introduce $r$-fold trivial extension algebras (see for example \cite[Section 2.2]{CDIM25} for more information and \cite[Chapter 2]{Hap88} for the related repetitive algebra) and give their basic properties and then use them to prove our main result.

\begin{definition}\label{def:T3}
Let $r\geq2$. The \emph{$r$-fold trivial extension} of $A$ is
$T_r(A)=A^r\oplus(DA)^r$, with multiplication
\begin{equation}\label{eq:T3product}
(a_i,\varphi_i)_{i=1}^r(b_i,\psi_i)_{i=1}^r
=\bigl(a_i b_i,\ a_i\psi_i+\varphi_i b_{i+1}\bigr)_{i=1}^r.
\end{equation}
Indices are read modulo $r$. In particular,
\begin{equation}\label{eq:T3matrix}
T_2(A)=\begin{pmatrix}A&DA\\DA&A\end{pmatrix},
\qquad
T_3(A)=\begin{pmatrix}A&DA&0\\0&A&DA\\DA&0&A\end{pmatrix}.
\end{equation}
Products of two off-diagonal entries are defined to be zero.
\end{definition}

The matrix notation does not prescribe products of two
dual elements; such products are zero by definition.

Put $\Gamma=T_r(A)$ and let $e_i$ denote the $i$th diagonal identity.
These idempotents are pairwise orthogonal and sum to $1$.
Every corner $e_i\Gamma e_i$ is identified with the same algebra $A$.
$T_r(A)$ is an Artin $R$-algebra and
for a right $\Gamma$-module $Z$ one has the direct sum of $R$-modules
$Z=\bigoplus_{j=1}^r Ze_j$. We call $Ze_j$ its $j$th component.
Note that multiplication
by an element of $e_i\Gamma e_{i+1}=DA$ can send the $i$th component
to the $(i+1)$st. The functor $Z\mapsto Ze_j$ is exact, since taking
the image of an idempotent preserves exact sequences.

\begin{lemma}\label{lem:nakayama}
Define an automorphism $\rho$ of $\Gamma$ by
\[
\rho((a_i,\varphi_i)_i)=(a_{i-1},\varphi_{i-1})_i.
\]
Thus $\rho(e_i)=e_{i+1}$. The notation ${}_1\Gamma_\rho$ means
the underlying $R$-module $\Gamma$ with its usual left action and right
action $x\cdot b=x\rho(b)$. There is an isomorphism of $\Gamma$-bimodules
\begin{equation}\label{eq:DGamma}
D\Gamma\cong {}_1\Gamma_\rho.
\end{equation}
In particular, $\Gamma$ is selfinjective. If $Z_\rho$ denotes the right $\Gamma$-module with action $z\cdot a=z\rho(a)$, then
\begin{equation}\label{eq:nutwist}
\nu_\Gamma Z\cong Z_\rho
\end{equation}
naturally in $Z$.
\end{lemma}

\begin{proof}
The formula for multiplication immediately shows that $\rho$ is an algebra automorphism; its inverse shifts all indices backwards. Define the $R$-linear map $\ell:\Gamma\to E$ by
\[
\ell((a_i,\varphi_i)_i)=\sum_{i=1}^r\varphi_i(1).
\]
For $x=(a_i,\varphi_i)_i$ and $y=(b_i,\psi_i)_i$, the multiplication formula \eqref{eq:T3product} gives
\begin{align}
\ell(xy)
&=\sum_i\bigl(\psi_i(a_i)+\varphi_i(b_{i+1})\bigr),\label{eq:pairing}\\
\ell(y\rho(x))
&=\sum_i\bigl(\varphi_{i-1}(b_i)+\psi_i(a_i)\bigr)
=\ell(xy).\label{eq:nakidentity}
\end{align}
Consider the $R$-linear map
\[
\theta:\Gamma\longrightarrow D\Gamma,
\qquad \theta(x)(y)=\ell(yx).
\]
Suppose $\theta(x)=0$. In the formula
\[
\ell(yx)=\sum_i\bigl(\varphi_i(b_i)+\psi_i(a_{i+1})\bigr),
\]
first vary one $b_i$ at a time and then one $\psi_i$ at a time. This gives $\varphi_i=0$ and $a_{i+1}=0$ for all $i$, since $E$ is a cogenerator for finitely generated $R$-modules. Hence $\theta$ is injective. Since $D$ preserves finite $R$-length, $\theta$ is bijective.

The dual bimodule action on $D\Gamma$ is $(a f b)(y)=f(bya)$. Therefore
\[
a\theta(x)=\theta(ax),
\qquad
\theta(x)b=\theta(x\rho(b)).
\]
Thus $\theta$ identifies $D\Gamma$ with ${}_1\Gamma_\rho$, proving \eqref{eq:DGamma}. Since $D\Gamma$ is injective as a right module and twisting by an automorphism is an exact autoequivalence, $\Gamma$ is selfinjective. Finally,
\[
\nu_\Gamma Z\cong Z\otimes_\Gamma D\Gamma
\cong Z\otimes_\Gamma {}_1\Gamma_\rho\cong Z_\rho,
\]
which proves \eqref{eq:nutwist}.
\end{proof}

For $i=1,\ldots,r$, put $F_i=-\otimes_A e_i\Gamma$.
For $M\in\modu A$, the underlying $R$-module of $F_iM$
is $M\oplus(M\otimes_A DA)$, and multiplication by
$\gamma=(a_j,\varphi_j)_j$ is
\begin{equation}\label{eq:induced-action}
(m,n)\gamma=(ma_i,\ m\otimes\varphi_i+na_{i+1}).
\end{equation}
Thus the first summand sits at $e_i$ and the second at $e_{i+1}$.

\begin{lemma}\label{lem:components}
For every $M\in\modu A$ one has
\begin{equation}\label{eq:components}
(F_iM)e_i\cong M,\qquad
(F_iM)e_{i+1}\cong M\otimes_A DA,\qquad
(F_iM)e_j=0\quad(j\notin\{i,i+1\}).
\end{equation}
Moreover, there is a natural isomorphism
\begin{equation}\label{eq:nushift}
\nu_\Gamma(F_iM)\cong F_{i-1}M.
\end{equation}
\end{lemma}
\begin{proof}
All indices are read modulo $r$. Put $N=M\otimes_A DA$.
For each index $k$, the only nonzero components of
$e_k\Gamma$ are
\[
e_k\Gamma e_k=A,
\qquad
e_k\Gamma e_{k+1}=DA.
\]
Thus $e_k\Gamma\cong A\oplus DA$ as a left $A$-module,
and we have canonical identifications of $R$-modules
\[
F_kM
=M\otimes_A e_k\Gamma
\cong (M\otimes_A A)\oplus(M\otimes_A DA)
\cong M\oplus N.
\]
Here the last identification uses $m\otimes a\mapsto ma$.
Under these identifications, \eqref{eq:induced-action} reads
\[
(m,n)\gamma
=
(ma_k,\ m\otimes\varphi_k+na_{k+1}),
\qquad
\gamma=(a_j,\varphi_j)_j\in\Gamma.
\]

Taking $k=i$ and multiplying by the diagonal idempotents gives
\[
(m,n)e_i=(m,0),
\qquad
(m,n)e_{i+1}=(0,n),
\qquad
(m,n)e_j=0
\quad\text{if }j\notin\{i,i+1\}.
\]
Moreover, the corner algebras $e_i\Gamma e_i=A$ and
$e_{i+1}\Gamma e_{i+1}=A$ act on these two components
by the usual right $A$-actions on $M$ and $N$, respectively.
This proves \eqref{eq:components} as a statement about
right $A$-modules.

For the Nakayama-functor assertion,
Lemma~\ref{lem:nakayama} gives a natural isomorphism
\[
\nu_\Gamma(F_iM)\cong(F_iM)_\rho.
\]
The twisted module $(F_iM)_\rho$ has the same underlying
$R$-module $M\oplus N$, but its action is
\[
(m,n)\mathbin{\star}\gamma=(m,n)\rho(\gamma).
\]
Since $\rho(\gamma)=(a_{j-1},\varphi_{j-1})_j$, the action
formula above yields
\[
(m,n)\mathbin{\star}\gamma
=
(ma_{i-1},\ m\otimes\varphi_{i-1}+na_i).
\]
This is exactly the action on $F_{i-1}M$, obtained by
taking $k=i-1$ in the same formula. Consequently, the
identity map on $M\oplus N$ is a $\Gamma$-module
isomorphism
\[
\alpha_M:(F_iM)_\rho\xrightarrow{\ \sim\ }F_{i-1}M.
\]

Finally, for any $A$-linear map $f:M\to M'$, both
$(F_i f)_\rho$ and $F_{i-1}f$ are represented under
these identifications by
\[
f\oplus(f\otimes_A\operatorname{id}_{DA}).
\]
Hence
\[
\alpha_{M'}\circ(F_i f)_\rho
=
F_{i-1}f\circ\alpha_M,
\]
so $\alpha_M$ is natural in $M$. Composing with the
natural isomorphism supplied by
Lemma~\ref{lem:nakayama} proves \eqref{eq:nushift}.
\end{proof}
\begin{proposition}\label{prop:transport}
Suppose $M\in\modu A$ satisfies $\Ext_A^j(M,A)=0$ for every $j>0$. Then for $i=1,\ldots,r$, every $Z\in\modu\Gamma$ and every $n\geq0$,
\begin{equation}\label{eq:transport}
\Ext_\Gamma^n(F_iM,Z)\cong\Ext_A^n(M,Ze_i).
\end{equation}
In particular, if $M$ is selforthogonal and $X=F_1M$, then
$\Ext_\Gamma^n(X,X)=0$ for every $n>0$. If also $r\geq3$, then
\begin{equation}\label{eq:twovanish}
\Ext_\Gamma^n(\nu_\Gamma X,X)=0\qquad(n\geq0).
\end{equation}
\end{proposition}

\begin{proof}
As a left $A$-module, $e_i\Gamma\cong A\oplus DA$. Let $P_\bullet\to M$ be a projective resolution. Tensor--Hom adjunction and double duality identify the cochain complexes
\[
D(P_\bullet\otimes_A DA)
\cong\Hom_A(P_\bullet,DDA)
\cong\Hom_A(P_\bullet,A).
\]
The functor $D$ is exact and faithful, so
\[
D\Tor_j^A(M,DA)\cong\Ext_A^j(M,A)=0
\qquad(j>0).
\]
It follows that $\Tor_j^A(M,e_i\Gamma)=0$ for every $j>0$. Lemma~\ref{lem:induction} proves \eqref{eq:transport}.

Now let $X=F_1M$ and suppose that $M$ is selforthogonal. Since $Xe_1\cong M$, equation \eqref{eq:transport} gives
\[
\Ext_\Gamma^n(X,X)\cong\Ext_A^n(M,M)=0\qquad(n>0).
\]
Now assume $r\geq3$. Lemma~\ref{lem:components} gives $\nu_\Gamma X\cong F_rM$. Since $r$ differs from $1$ and $2$, we have $Xe_r=0$. Thus
\[
\Ext_\Gamma^n(\nu_\Gamma X,X)
\cong\Ext_\Gamma^n(F_rM,X)
\cong\Ext_A^n(M,Xe_r)=0
\qquad(n\geq0).
\]
This proves \eqref{eq:twovanish}.
\end{proof}

We first apply the two-fold extension to the second Tachikawa conjecture.
\begin{theorem}\label{thm:twofold}
Let $A$ be an Artin $R$-algebra. If $T_2(A)$ satisfies $(\TCtwo)$,
then $A$ satisfies $(\ARC)$. Consequently, global $(\TCtwo)$ implies
global $(\ARC)$.
\end{theorem}

\begin{proof}
Let $M$ satisfy $\Ext_A^i(M,M\oplus A)=0$ for every $i>0$.
Put $\Gamma=T_2(A)$ and $X=M\otimes_A e_1\Gamma$.
The algebra $\Gamma$ is selfinjective by Lemma~\ref{lem:nakayama}.
Proposition~\ref{prop:transport} gives
\[
\Ext_\Gamma^i(X,X)\cong\Ext_A^i(M,M)=0\qquad(i>0).
\]
Thus $(\TCtwo)$ for $\Gamma$ implies that $X$ is projective.
By Lemma~\ref{lem:induction}, $M$
is projective. 
\end{proof}

The three-fold extension gives the corresponding implication for the first Tachikawa conjecture.
\begin{theorem}\label{thm:main}
Let $A$ be an algebra and $M\in\modu A$ satisfy
\[
\Ext_A^i(M,M\oplus A)=0\qquad(i>0).
\]
Set $\Gamma=T_3(A)$, $X=M\otimes_A e_1\Gamma$ and
\[
B_M=\End_\Gamma(\Gamma\oplus X).
\]
Then $\domdim B_M=\infty$ and $\Ext_{B_M}^i(DB_M,B_M)=0$ for every $i>0$. If $B_M$ satisfies $(\TCone)$, then $M$ is projective. Consequently, global $(\TCone)$ implies global $(\ARC)$.
\end{theorem}

\begin{proof}
By Lemma~\ref{lem:nakayama}, $\Gamma$ is selfinjective. Proposition~\ref{prop:transport} gives
\[
\Ext_\Gamma^i(X,X)=\Ext_\Gamma^i(\nu_\Gamma X,X)=0
\qquad(i>0).
\]
Put $V=\Gamma\oplus X$. The mixed summands with $\Gamma$ in $\Ext_\Gamma^i(V,V)$ vanish because $\Gamma$ is projective and injective. Similarly, in $\Ext_\Gamma^i(\nu_\Gamma V,V)$ the summands with $\nu_\Gamma\Gamma=D\Gamma$ in the first variable vanish because $D\Gamma$ is projective, and the summands with $\Gamma$ in the second variable vanish because $\Gamma$ is injective. Hence
\[
\Ext_\Gamma^i(V,V)=\Ext_\Gamma^i(\nu_\Gamma V,V)=0
\qquad(i>0).
\]
Proposition~\ref{prop:endExt} now gives
\[
\Ext_{B_M}^i(DB_M,B_M)
\cong\Ext_\Gamma^i(\nu_\Gamma V,V)=0\qquad(i>0).
\]
Since $\Gamma$ is selfinjective, $V$ is a generator-cogenerator, and the last assertion of Proposition~\ref{prop:endExt} also gives $\domdim B_M=\infty$.
If $B_M$ satisfies $(\TCone)$, then it is selfinjective. Lemma~\ref{lem:splitgenerator}, applied to the generator-cogenerator $V$ over $\Gamma$, implies that $V$ is projective. Therefore $X$ is projective. Finally, $F_1$ reflects projectivity by Lemma~\ref{lem:induction}, so $M$ is projective. 
\end{proof}
We give now the proof of our first main theorem:
\begin{theorem}
Fix a commutative artinian ring $R$. The following conjectures are
globally equivalent for Artin $R$-algebras:
\begin{enumerate}
\item the Auslander--Reiten conjecture;
\item the generalised Nakayama conjecture;
\item the Nakayama conjecture;
\item the first Tachikawa conjecture;
\item the second Tachikawa conjecture.
\end{enumerate}
\end{theorem}
\begin{proof}
That (1) implies (2), (2) implies (3) and (3) implies (4) and (5) is well known, see for example \cite[Theorem 3.4.3]{Yam96}. In theorem \ref{thm:twofold} we proved that (5) implies (1) and in theorem \ref{thm:main} we proved that (4) implies (1) and thus all equivalences are proven.

\end{proof}

\section{Tilting equivalences and the Chen--Xi conjecture}
\label{sec:tilting}

In this section algebras are finite dimensional over a fixed field $K$,
and $D=\Hom_K(-,K)$.

A \emph{Morita algebra} is an algebra of the
form $\End_S(V)$, where $S$ is selfinjective and $V$ is a
generator-cogenerator; see \cite{KY13}.
A \emph{tilting module} $U_A$ has finite projective dimension
$n=\pdim_AU$, satisfies $\Ext_A^i(U,U)=0$ for all $i>0$, and admits
an exact sequence
\[
0\longrightarrow A\longrightarrow U_0\longrightarrow\cdots
\longrightarrow U_n\longrightarrow0,
\qquad U_j\in\add U.
\]
We write $\Db(\modu A)$ for the bounded derived category. For background on derived categories and tilting theory, see Krause~\cite{Kra22} and Zimmermann~\cite{Zim14}.
For a tilting module $U_A$, put $B=\End_A(U)$, so ${}_BU_A$ is a bimodule.
The tilting theorem gives quasi-inverse equivalences
\[
\mathbf R\Hom_A(U,-):\Db(\modu A)\longrightarrow\Db(\modu B),
\qquad
-\otimes_B^{\mathbf L}U:\Db(\modu B)\longrightarrow\Db(\modu A).
\]

We recall the uniqueness part of the Morita--Tachikawa correspondence needed below; see \cite[Theorem~2 and the following remark, p.~401]{Mue68} and \cite[Theorem~2]{Cru22}.
Every algebra $A$ with $\domdim A\geq2$ has a presentation
$A\cong\End_S(V)$ with $V$ a generator-cogenerator, and $S$ is
determined up to Morita equivalence.
Moreover, M\"uller's formula \cite[Lemma~3, pp.~401--402]{Mue68} gives
\begin{equation}\label{eq:mueller-field}
\domdim\End_S(V)
 =1+\inf\{i\geq1:\Ext_S^i(V,V)\ne0\},
\qquad\inf\varnothing=\infty.
\end{equation}
We also note that $\domdim A=\domdim A^{\op}$
\cite[Theorem~4, pp.~402--403]{Mue68}.
For a module $N$, write $\Omega_A^{-1}N$ for the cokernel of its
embedding into an injective envelope. The following construction is
the first canonical tilting module considered in \cite{CX16}. For the convenience of the reader we give the elementary proof.
\begin{lemma}\label{lem:canonical}
Suppose $\domdim A\geq1$, denote by $I=I^0(A)$ the injective envelope of $A$ and $C=\Omega_A^{-1}A$.
Then $U=I\oplus C$ is a tilting module of projective dimension at
most one. If $A$ is not selfinjective, its projective dimension is one.
\end{lemma}
\begin{proof}
The defining exact sequence
\[
0\longrightarrow A\longrightarrow I\longrightarrow C\longrightarrow0
\]
has $I$ projective, so $\pdim_A C\leq1$.
Injectivity of $I$ gives $\Ext_A^1(C,I)=0$. We have
$\Ext_A^1(C,C)\cong \Ext_A^1(C,\Omega^{-1}(A)) \cong \Ext_A^2(C,A)=0$.
Projectivity of $I$ and $\pdim_A C\leq1$ now imply
$\Ext_A^i(U,U)=0$ for all $i>0$. The displayed sequence has its
last two terms in $\add U$, and hence supplies the required tilting
coresolution. If $\pdim_A U=0$, then $C$ is projective, the sequence
splits, and $A$ is a direct summand of the injective module $I$.
Thus $A$ is selfinjective. Conversely, when $A$ is selfinjective,
$I=A$, $C=0$, and $U=A$ is a tilting module of projective dimension zero.
\end{proof}

We use the following consequences of Chen and Xi's results.
\begin{proposition}[Chen--Xi]\label{prop:chen-xi-input}
\begin{enumerate}
\item If $\domdim A\geq2$ and $I=I^0(A)$, then $A$ is a Morita
algebra if and only if $\add I=\add\nu_A I$.
\item If $\domdim A\geq2$ and $A$ is not a Morita algebra, the
canonical tilting module of Lemma~\ref{lem:canonical} satisfies
$\domdim\End_A(I\oplus\Omega_A^{-1}A)\leq1$.
\end{enumerate}
\end{proposition}

\begin{proof}
Part~(1) follows from \cite[Lemmas~3.4(1) and~3.8]{CX16}.
Part~(2) is \cite[Proposition~4.17(1)]{CX16} for the first canonical
tilting module.
\end{proof}
As before, we denote by $\NC(A)$ the statement that the Nakayama conjecture is true for the algebra $A$, that is $\domdim A= \infty$ implies that $A$ is selfinjective.

\begin{theorem}\label{thm:tilting-note}
The following statements, quantified over all finite dimensional
$K$-algebras, are equivalent.
\begin{enumerate}[label=\textup{(\Alph*)}]
\item\label{cond:tilting}
For every tilting module $U_A$, with $B=\End_A(U)$, the statements
$\NC(A)$ and $\NC(B)$ are equivalent.
\item\label{cond:morita}
Every algebra of infinite dominant dimension is a Morita algebra.
\item\label{cond:tc1}
The first Tachikawa conjecture holds for every algebra.
\end{enumerate}
\end{theorem}

\begin{proof}
First assume \ref{cond:tc1}. If $\domdim A=\infty$, choose
$A\cong\End_S(V)$ by Morita--Tachikawa with a generator-cogenerator $S$-module $V$. Formula~\eqref{eq:mueller-field} gives
$\Ext_S^i(V,V)=0$ for every $i>0$. Since $S,DS\in\add V$, it follows
that $\Ext_S^i(DS,S)=0$ for every $i>0$. By the first Tachikawa
conjecture, $S$ is selfinjective. Hence $A$ is Morita, proving
\ref{cond:morita}.

Conversely, assume \ref{cond:morita} and suppose
$\Ext_S^i(DS,S)=0$ for every $i>0$. Put $V=S\oplus DS$.
Projectivity of $S$, injectivity of $DS$, and the assumed vanishing
give $\Ext_S^i(V,V)=0$ for all $i>0$. Thus $A=\End_S(V)$ has infinite
dominant dimension by \eqref{eq:mueller-field}, and is Morita by
\ref{cond:morita}. It therefore also has a presentation as the
endomorphism algebra of a generator-cogenerator over a selfinjective
algebra. Uniqueness in the Morita--Tachikawa correspondence makes $S$
Morita equivalent to that selfinjective algebra. Since selfinjectivity
is Morita invariant, $S$ is selfinjective. This proves \ref{cond:tc1}.

Next assume \ref{cond:tilting}, and let $\domdim A=\infty$.
If $A$ is selfinjective it is already a Morita algebra, using the
presentation $A\cong\End_A(A)$. Otherwise,  the canonical tilting module $U=I^0(A)\oplus\Omega_A^{-1}A$ is of projective
dimension one. If $A$ were not Morita,
Proposition~\ref{prop:chen-xi-input}(2) would give
$\domdim\End_A(U)\leq1$.
Hence $\NC(\End_A(U))$ holds vacuously, and \ref{cond:tilting} gives
$\NC(A)$. As $\domdim A=\infty$, the algebra $A$ is then
selfinjective, and therefore Morita, a contradiction. This proves
\ref{cond:morita}.

Finally assume \ref{cond:morita}. We have proved that this implies
\ref{cond:tc1}. Theorem~\ref{mainresult1}, applied with $R=K$,
then gives the Nakayama conjecture for all finite dimensional
$K$-algebras.
Thus, for any tilting module $U_A$ and $B=\End_A(U)$, both $\NC(A)$
and $\NC(B)$ are true. In particular they are equivalent, proving
\ref{cond:tilting}.
\end{proof}

Chen and Xi formulated their conjecture for finite dimensional algebras
over a field \cite[\S5.2, conjecture~(1)]{CX16}; see also
\cite[Conjecture~6.5]{Xi18}. More precisely, their assertion is
\begin{equation}\label{eq:CX-conjecture}
\Db(\modu A)\simeq\Db(\modu B)
\quad\Longrightarrow\quad
\bigl(\domdim A<\infty\Longleftrightarrow\domdim B<\infty\bigr).
\end{equation}
We can now prove Theorem~\ref{mainresult2} in the following expanded form.

\begin{theorem}\label{thm:derived}
For finite dimensional algebras over a fixed field $K$, the following
statements are equivalent:
\begin{enumerate}
\item the Nakayama conjecture holds for every algebra;
\item the first Tachikawa conjecture holds for every algebra;
\item the Nakayama implication $\NC(A)$ is invariant under derived
      equivalences;
\item the Chen--Xi conjecture \eqref{eq:CX-conjecture} holds.
\end{enumerate}
Each is also equivalent to the conditions of
Theorem~\ref{thm:tilting-note}.
\end{theorem}

\begin{proof}
The equivalence of (1) and (2) is part of Theorem~\ref{mainresult1}, with $R=K$.
Statement (1) implies (3) immediately. Statement (3) implies
condition~\ref{cond:tilting} of Theorem~\ref{thm:tilting-note}, since
tilting equivalences are derived equivalences; that theorem then
gives (2).

Under (1), an algebra has infinite dominant dimension exactly when it
is selfinjective. Since selfinjectivity is preserved under derived
equivalences by \cite[Corollary~5.3]{XZ25}, this proves (4).
Conversely, assume (4). If an algebra $A$ of infinite dominant dimension
were not Morita, the canonical tilting module as in the proof of
Theorem~\ref{thm:tilting-note} would produce a derived equivalent
algebra of dominant dimension at most one. This contradicts (4).
Thus condition~\ref{cond:morita} of that theorem holds, giving (2),
and hence (1).
\end{proof}

As a final remark, we note that we worked in this section with algebras over field instead of general Artin algebras, since we used that being selfinjective is closed under derived equivalences, which is to our knowledge only established for algebras over fields in \cite{XZ25}.
The conjecture of Chen-Xi is stated for algebras over fields.

\section*{Statement on the use of AI}
ChatGPT Astra was used in author-directed research and manuscript
preparation, including the development of auxiliary arguments,
expansion of proof details, literature searches, consistency checks,
and \LaTeX{} editing. The authors
are responsible for the mathematical results, proofs, citations, and
final text.

\end{document}